\documentclass[12pt,final]{amsart}
\usepackage{amsmath,amssymb}

\usepackage{cite}

\usepackage{color}
\usepackage{soul,graphicx}

\usepackage[color]{showkeys}
\definecolor{refkey}{rgb}{0,0,1}
\definecolor{labelkey}{rgb}{1,0,0}

\newcommand{\ol}{\overline}

\newcommand{\tht}{\theta}

\newcommand{\eq} [1] {\begin{equation}\label{#1}\quad}
\newcommand{\en} {\end{equation}}

\newcommand{\Tr}{\mathbb{\rm Tr}}
\newcommand{\Cof}{\mathop{\rm Cof}}

\newcommand{\tm}{\times}

\newcommand{\iy}{\infty}

\newcommand{\bT}{\mathbb{T}}
\newcommand{\bC}{\mathbb{C}}
\newcommand{\bP}{\mathbb{P}}
\newcommand{\bD}{\mathbb{D}}

\newcommand{\calA}{\mathcal{A}}

\newcommand{\calI}{\mathcal{I}}

\newtheorem{theorem}{\bf  Theorem}

\begin{document}
\begin{center}
{\bf On matrix generalization of Robinson's Energy Delay Theorem}\\[2mm]
{ L. Ephremidze$^a$ and W. H. Gerstacker$^b$}\\[5mm]
{\footnotesize $^a$ New York University Abu Dhabi, UAE, and  A. Razmadze Mathematical Institute, Georgia\\
$^b$ Institute for Mobile Communications, Erlangen, Germany \par}
\end{center}
\vskip+0.5cm
{\small{\bf Abstract.} An elementary proof of Robinson's Energy Delay Theorem on minimum-phase functions is presented. The proof is generalized to the matrix case as well which turns out to be simpler than the earlier one proposed in \cite{WG} for polynomial matrix functions.

\vskip+0.5cm
{\bf Keywords:} Minimum-phase functions; Hardy spaces
\par}

\section{Introduction}

Let $\bD$ be the unit disk in the complex plane and $\bT=\{z\in\bC:|z|=1\}$ be its boundary. The set of analytic in $\bD$ functions is denoted by $\calA(\bD)$. Consider the subset of functions
$$
f(z)=\sum_{n=0}^\iy a_nz^n\in\calA(\bD),
$$
satisfying
$$
\sum_{n=0}^\iy|a_n|^2<\iy.
$$
Such functions are known as $z$-{\em transforms} (resp. {\em transfer functions}) of discrete-time {\em causal signals} (resp. {\em filter impulse responses}) with a finite energy in engineering, while the subset of functions is called the Hardy space $H_2=H_2(\bD)$ in mathematics. It is well known that the boundary values of $f\in H_2$ exists a.e.,
\begin{equation}\label{01}
f^*(e^{i\tht})=\lim_{r\to 1-}f(re^{i\tht})\;\;\text{ for a.a.} \tht\in[0,2\pi)
\end{equation}
and $f^*\in L_2(\bT)$, the Lebesgue space of square integrable functions on $\bT$. (Actually, there is a one-to-one correspondence between the functions from the Hardy space $H_2$ and their boundary value functions, i.e. if $f,g\in H_2$ and $f^*(e^{i\tht})=g^*(e^{i\tht})$ for a.e. $\tht$, then $f=g$.) Furthermore, the so-called {\em Paley-Wiener condition} is satisfied:
\begin{equation}\label{02}
\int_0^{2\pi}\log|f^*(e^{i\tht})|\,d\tht >-\iy\,.
\end{equation}

For any function $f\in H_2$, the inequality
\begin{equation}\label{21}
|f(0)|\leq \exp\left(\frac1{2\pi}\int_0^{2\pi}\log|f^*(e^{i\tht})|\,d\tht\right)
\end{equation}
is valid (see, e.g., \cite[Th. 17.17]{Ru}) and those extreme functions for which the equality holds in \eqref{21} are called {\em outer} in mathematics and of {\em minimum-phase} in engineering (also the term {\em optimal} can optionally be used in both). Such functions play an important role in mathematics as well as in engineering. However, the original definition of outer functions, introduced by Beurling \cite{Beu} differs from the above definition and says that the representation
\begin{equation}\label{215}
 f(z)=c\cdot\exp\left(\frac 1{2\pi}
\int\nolimits_0^{2\pi}\frac{e^{i\theta}+z}{e^{i\theta}-z}\log
|f^*(e^{i\tht})|\,d\theta\right), \;\;\;\text{ where }\;|c|=1,
\end{equation}
holds for outer functions. This representation easily implies that the equality holds in \eqref{21} for outer functions and it can be proved that the converse is also true. Furthermore, Beurling \cite{Beu} proved that every $h\in H_2$ can be factorized as
\begin{equation}\label{22}
 h(z)=B(z)\calI(z)f(z),
\end{equation}
where $B(z)=\prod_{n=1}\frac{|z_n|}{z_n}\,\frac{z_n-z}{1-\ol{z}_nz}$ is a Blaschke product, $\calI$ is a bounded analytic function without zeros inside $\bD$, which satisfies $|\calI^*(e^{i\tht})|=1$ for a.e. $\tht$ (such functions are called {\em singular inner} functions, if we exclude the requirement $\calI(z)\not=0$ for $z\in\bD$, then we get the definition of {\em inner} function, i.e. $B\calI$ is an inner function), and $f$ is an outer function. (Observe that $|h^*(e^{i\tht})|=|f^*(e^{i\tht})|$ a.e.) In these terms, a function is outer if and only if an inner factor does not exist in factorization \eqref{22}, i.e. $B\equiv\calI\equiv 1$.

These definitions and factorization \eqref{22} are now classical in mathematical theory of Hardy spaces. However engineers frequently discard the middle term in the factorization \eqref{22} (a singular inner factor never exists for rational functions, namely, it has the form $ \calI(z)=\exp\big(\frac 1{2\pi}
\int\nolimits_0^{2\pi}\frac{e^{i\theta}+z}{e^{i\theta}-z}\,d\mu_s(\theta)\big)$, where $\mu_s$ is a singular measure on $[0,2\pi)$, which never encounters in practise) and define a minimum-phase function $f\in H_2(\bD)$ by the condition $1/f\in\calA(\bD)$ (i.e. $f(z)\not=0$ for $z\in\bD$). This definition can be used for rational functions, however, not for arbitrary analytic functions. As an example of the singular inner function $\calI$ shows, the inequality in \eqref{21} might be strict in this case ($|\calI(0)|<1$, while $\int_0^{2\pi}\log|\calI^*(e^{i\tht})|\,d\tht=0$) and the equality may not hold in \eqref{21} as it is incorrectly claimed in \cite[p.574]{PP}.

Another important property of minimum-phase functions was introduced by Robinson \cite{Rob}. Namely he proved the following

\begin{theorem}
Let $f(z)=\sum_{n=0}^\iy a_nz^n$ and $g(z)=\sum_{n=0}^\iy b_nz^n$ be functions from $H_2$ satisfying $|f^*(e^{i\tht})|=|g^*(e^{i\tht})|$ for a.e. $\tht$. If $f$ is of minimum-phase, then for each $N$,
\begin{equation}\label{31}
\sum_{n=0}^N |a_n|^2\geq \sum_{n=0}^N |b_n|^2.
\end{equation}
\end{theorem}

Robinson gave a physical interpretation to inequality \eqref{31} ``that among all filters with the same gain, the outer filter makes the energy built-up as large as possible, and it does so for every positive time" \cite{Rob1} and found geological applications of minimum-phase waveforms. Therefore, the term {\em minimum-delay} \cite[p. 211]{kai99} functions can be equivalently used for optimal functions and Theorem 1 is known as Energy Delay Theorem within a geological community \cite[p. 52]{Clae}

Theorem 1 was further extended to polynomial matrix case and used in MIMO communications in \cite{WG}. We formulate this result in general matrix form after introducing corresponding notations.

In the present paper, we provide a very simple proof of Theorem 1 based on classical facts from the theory of Hardy spaces. We give a matrix generalization of this proof as well in Section 4. Some engineering applications of Theorems 1 and 2 can be found in \cite{Clae} and \cite{WG}.

\section{Notation}

Let $L_p=L_p(\bT)$, $0<p\leq\iy$, be the Lebesgue space of $p$-integrable complex functions $f^*$ with the usual norm $\|f^*\|_{L_p}=\big(\int_0^{2\pi}|f^*(e^{i\tht})|^p\,d\tht\big)^\frac{1}{p}$ for $p\geq 1$ (with standard modification for $p=\infty$), and let $H_p=H_p(\bD)$, $0<p\leq\iy$, be the Hardy space
$$
\left\{f\in\mathcal{A}(\mathbb{D}):\sup\limits_{r<1}
\int\nolimits_0^{2\pi}|f(re^{i\theta})|^p\,d\theta<\infty\right\}
$$
with the norm $\|f\|_{H_p}=\sup_{r<1}\|f(re^{i\cdot})\|_{L_p}$ for $p\geq 1$ ($H_\infty$ is the space of bounded analytic functions with the supremum norm). It is well-known that boundary values function $f^*$ (see \eqref{01}) exists for every $f\in H_p$, $p>0$, and belongs to $L_p$. Furthermore
\begin{equation}\label{N11}
\|f\|_{H_p}=\|f^*\|_{L_p}
\end{equation}
for every $p\geq 1$, and it follows from standard Fourier series theory that
\begin{equation}\label{N12}
\left\|\sum\nolimits_{n=0}^\iy a_nz^n\right\|_{H_2}=\left(\sum\nolimits_{n=0}^\iy|a_n|^2\right)^{1/2}.
\end{equation}

The condition \eqref{02} holds for every $f\in H_p$ and the function $f$ is called outer if the representation \eqref{215} is valid as well. We have the equality (the optimality condition) instead of the inequality in \eqref{21} if and only if $f$ is outer (see \cite[Th. 17.17]{Ru}). One can check, using the H\"{o}lder inequality, that if $f$ and $g$ are outer functions from $H_p$ and $H_q$, respectively, then the product $fg$ is the outer function from $H_{{pq}/(p+q)}$.

We repeat that $u\in\calA(\bD)$ is called a inner function if $u\in H_\iy$ and
\begin{equation}\label{u1}
|u^*(e^{i\tht})|=1\;\;\text{ for a.a. } \tht\in[0,2\pi).
\end{equation}

Now we consider matrices and matrix functions. $\bC^{d\tm d}$, $L_p^{d\tm d}$, etc., denote the set of $d\tm d$ matrices with entries from $\bC$, $L_p$, etc. The elements of $L_p^{d\tm d}$ (resp. $H_p^{d\tm d}$) are assumed to be matrix functions with domain $\bT$ (resp. $\bD$) and range $\bC^{d\tm d}$, and without any ambiguity we assume that $F^* \in L_p^{d\tm d}$ for $F\in H_p^{d\tm d}$.

For $M\in \bC^{d\tm l}$, we consider the Frobenius norm of $M$:
$$
\|M\|_2=\big(\sum\nolimits_{i=1}^d\sum\nolimits_{j=1}^l|m_{ij}|^2\big)^{1/2}=\big(\Tr(MM^H)\big)^{1/2},
$$
where $M^H=\ol{M}^T$, and for $F\in H_p^{d\tm d}$, we define
$$
\|F\|_{H_2^{d\tm d}}=\big(\sum\nolimits_{i=1}^d\sum\nolimits_{j=1}^d|f_{ij}|_{H_2}^2\big)^{1/2}.
$$
Similarly, we define $\|F^*\|_{L_2^{d\tm d}}$ for $F^* \in L_2^{d\tm d}$. By virtue of \eqref{N11}, we have
\begin{equation}\label{R1}
\|F\|_{H_2^{d\tm d}}=\|F^*\|_{L_2^{d\tm d}}
\end{equation}
and, similarly to \eqref{N12},
\begin{equation}\label{N21}
\left\|\sum\nolimits_{n=0}^\iy A_nz^n\right\|_{H_2^{d\tm d}}=\left(\sum\nolimits_{n=0}^\iy\|A_n\|_2^2\right)^{1/2}
\end{equation}
for any sequence of matrix coefficients $A_0, A_1,\ldots$ from $\bC^{d\tm d}$.

A matrix function $F\in H_2^{d\tm d}$ is called {\em outer}, if $\det F$ is an outer function from $H_{2/d}$. This definition is equivalent to number of other definitions of outer matrix functions (see, e.g., \cite{Ep10}). On the other hand, a matrix function $U\in \calA(\bD)^{d\tm d}$ is called {\em inner}, if $U\in H_\iy^{d\tm d}$ and $U^*$ is unitary a.e.,
\begin{equation}\label{U1}
U^*(e^{i\tht})\big(U^*(e^{i\tht})\big)^H=I_d  \;\;\text{ for a.a. } \tht\in[0,2\pi).
\end{equation}

We will make use of the following standard result from the theory of Hardy spaces (see \cite[p. 109]{Ko})

{\bf Smirnov's Generalized Theorem:} if $f=g/h$, where $g\in H_p$, $p>0$,  $h$ is an outer function from $H_q$, $q>0$, and $f^*\in L_r$, $r>0$, then $f\in H_r$.

\section{Proof of Theorem 1}

For a positive integer $N$, let $P_N$ be the projector operator on $H_2$ defined by
$$
P_N:\sum\nolimits_{n=0}^\iy a_nz^n \longmapsto \sum\nolimits_{n=0}^N a_nz^n.
$$
In order to prove Theorem 1, we need to show that (see \eqref{N12})
\begin{equation}\label{51}
\|P_N(f)\|_{H_2}\geq \|P_N(g)\|_{H_2}\,.
\end{equation}
For any bounded analytic function $u\in H_\iy$, we have
\begin{equation}\label{52}
P_N(uf)=P_N\big(u\cdot P_N(f)\big)
\end{equation}
since $P_N\big(u\cdot P_N(f)\big)=P_N\big(u(f-(f- P_N(f)))\big)=P_N(uf)-P_N\big(u(f- P_N(f))\big)=P_N(uf)$. Here we utilized that
$f- P_N(f)=\sum\nolimits_{n=N+1}^\iy a_nz^n$ and, because of analyticity of $u$, we have $u(f- P_N(f))=\sum\nolimits_{n=N+1}^\iy c_nz^n$ for some coefficients $c_{N+1}, c_{N+2},\ldots$

By virtue of Beurling factorization theorem \eqref{22}, there exits an inner function $u$ such that $g=uf$. Therefore, taking into account \eqref{N11}, \eqref{u1}, and \eqref{52}, we get
$$
\|P_N(f)\|_{H_2}=\|u\cdot P_N(f)\|_{H_2}\geq \|P_N\big(u\cdot P_N(f)\big)\|_{H_2}=\|P_N(uf)\|_{H_2}=\|P_N(g)\|_{H_2}.
$$
Thus \eqref{51} follows and Theorem 1 is proved.

\section{The matrix case}

In this section we prove the following

\begin{theorem}
Let $F(z)=\sum_{n=0}^\iy A_nz^n$, $A_n\in \bC^{d\tm d}$, and $G(z)=\sum_{n=0}^\iy B_nz^n$, $B_n\in \bC^{d\tm d}$, be matrix functions from $H_2^{d\tm d}$ satisfying
\begin{equation}\label{61}
F^*(e^{i\tht})\big(F^*(e^{i\tht})\big)^H=G^*(e^{i\tht})\big(G^*(e^{i\tht})\big)^H   \;\;\text{ for a.a. } \tht\in[0,2\pi).
\end{equation}
If $F$ is optimal, then for each $N$,
\begin{equation}\label{62}
\sum_{n=0}^N \|A_n\|_2^2\geq \sum_{n=0}^N \|B_n\|_2^2.
\end{equation}
\end{theorem}

\begin{proof}
Let $\bP_N$ be the projector operator on $H_2^{d\tm d}$ defined by
$$
\bP_N: \sum\nolimits_{n=0}^\iy A_nz^n\longmapsto \sum\nolimits_{n=0}^N A_nz^n.
$$
By virtue of \eqref{N21}, we have to prove that
\begin{equation}\label{62}
\|\bP_N(F)\|_{H_2^{d\tm d}}\geq \|\bP_N(G)\|_{H_2^{d\tm d}}
\end{equation}

Let
\begin{equation}\label{63}
U(z)=F^{-1}(z)G(z).
\end{equation}
It follows from \eqref{61} that \eqref{U1} holds. Therefore $U^*\in L_\iy^{d\tm d}$.
Since, in addition, $F^{-1}(z)=\frac1{\det F(z)}\Cof\big(F(z)\big)$, where $\det F(z)$ is an outer function, by the generalized Smirnov's theorem (see Sect. 2), we have $U \in H_\iy^{d\tm d}$.

Exactly in the same manner as \eqref{52} was proved, we can show that
\begin{equation}\label{635}
\bP_N(FU)=\bP_N\big( \bP_N(F)\cdot U\big).
\end{equation}
Since  unitary transformations preserve standard Euclidian norm of the space $\bC^d$, it follows from \eqref{U1} that, for any $V\in\bC^{1\tm d}$,
\begin{equation}\label{65}
\|V\|_2=\|V\cdot U^*(e^{i\tht})\|_2 \;\;\text{ for a.a. } \tht\in[0,2\pi).
\end{equation}
Therefore, by virtue of \eqref{R1} and \eqref{65},
\begin{equation}\label{66}
\|X\|_{H_2^{d\tm d}}=\|X^*\|_{L_2^{d\tm d}}=\|X^*U^*\|_{L_2^{d\tm d}}=\|XU\|_{H_2^{d\tm d}}
\end{equation}
for any $X\in H_2^{d\tm d}$. It follows now from \eqref{66}, \eqref{635}, and \eqref{63}   that
$$
\|\bP_N(F)\|_{H_2^{d\tm d}}=\|\bP_N(F)\cdot U\|_{H_2^{d\tm d}}\geq \|\bP_N\big(\bP_N(F)\cdot U\big)\|_{H_2^{d\tm d}}=
\|\bP_N(FU)\|_{H_2^{d\tm d}}=\|\bP_N(G)\|_{H_2^{d\tm d}}
$$
Thus \eqref{62} is true and Theorem 2 is proved.
\end{proof}

\section{Acknowledgments}

The first author would like to express his gratitude towards Professor Anthony Ephremides for indicating him several practical applications of minimum-phase signals, which naturally arises in mathematics via spectral factorization, in communication.

The manuscript was ready for submission when Professor Ilya Spitkovsky informed us that Robinson's Energy Delay Theorem is proved also in \cite{Goh-94}  for general operator valued functions in abstract Hilbert spaces. The authors are grateful to him for indicating this reference.

\def\cprime{$'$}
\providecommand{\bysame}{\leavevmode\hbox to3em{\hrulefill}\thinspace}
\providecommand{\MR}{\relax\ifhmode\unskip\space\fi MR }
\providecommand{\MRhref}[2]{%
  \href{http://www.ams.org/mathscinet-getitem?mr=#1}{#2}
}
\providecommand{\href}[2]{#2}

\end{document}